\documentclass{amsart}

\usepackage[utf8]{inputenc}
\usepackage[english]{babel}
\usepackage{amsmath,amssymb,amsthm}
\usepackage{lmodern}
\newtheorem{theorem}{Theorem}
\newtheorem{lemma}{Lemma}
\newtheorem{proposition}{Proposition}

\newcommand{\bbC}{\mathbb{C}}
\newcommand{\bbR}{\mathbb{R}}
\newcommand{\bbN}{\mathbb{N}}
\newcommand{\bbZ}{\mathbb{Z}}
\newcommand{\bbQ}{\mathbb{Q}}

\usepackage{lmodern}
\DeclareMathAlphabet{\mathbbo}{U}{bbold}{m}{n}
\newcommand{\1}{\mathbbo{1}}
\newcommand{\bbE}{\mathbb{E}}
\newcommand{\floor}[1]{\lfloor #1 \rfloor}
\newcommand{\bbP}{\mathbb{P}}
\DeclareMathOperator{\Reelle}{Re}
\DeclareMathOperator{\Imaginaire}{Im}

\begin{document}

\title{A simple Master Theorem for Discrete Divide and Conquer Recurrences}

{
\author{Olivier Garet}
\address{Université de Lorraine, CNRS, IECL, F-54000 Nancy, France}
\email{Olivier.Garet@univ-lorraine.fr}
}

\def\motsclefs{Divide-and-conquer recurrence, Dirichlet series, Tauberian theorem.}



\begin{abstract}
   The aim of this note is to provide a Master Theorem for some
  discrete divide and conquer recurrences:
  \[X_{n}=a_n+\sum_{j=1}^m b_j X_{\floor{\frac{n}{m_j}}},\]
  where the $m_i$'s are integers with $m_i\ge 2$.
  The main novelty of this work is there is no assumption of regularity or monotonicity for $(a_n)$.
  Then, this result can be applied to various sequences of random variables $(a_n)_{n\ge 0}$, for example
  such that $\sup_{n\ge 1}\bbE(|a_n|)<+\infty$.
  \end{abstract}
.
\maketitle

\begin{center}\emph{
  Published by the North-Western European Journal of Mathematics,\\
  n° 8 (2022) p. 91--100}
\end{center}  

\section{Introduction}
Divide-and-conquer methods are widely used in Computer Science.
The analysis of the cost of the algorithm naturally leads to divide-and-conquer recurrences.
The methods to study these recurrences are  popularized as ``Master theorems'' in the literature
of Computer Science. See e.g. the reference books by Cormen et al~\cite{cormen} or Goodrich and Tamassia~\cite{goodrich}.

In the sequel, we consider  sequences $(X_n)_{n\ge 0}$ that are defined by
$X_0=a_0$, then
\begin{align}
  \label{recurr}
  X_{n}&=a_n+\sum_{j=1}^m b_j X_{\floor{\frac{n}{m_j}}},
\end{align}
where the $m_i$'s are integer with $m_i\ge 2$ and $\floor{x}$ denotes the only $n\in\mathbb{Z}$ such that $x-n\in [0,1)$. 

Of course, in Computer Science, $a_n$ and $X_n$ represent  computation times and are therefore positive.
However, the case of negative $a_n$ and $X_n$ can be of theoretical interest.

In the literature of Computer Science, $(a_n)$ is supposed to be deterministic.
Nevertheless, in the context of randomized algorithm, eventually involving Monte-Carlo simulation, it is natural
to consider the case of a random $(a_n)$ and observe the fluctuations of the computation time.

One of the most general results in the field of Computer Science is due to Akra and Bazzi~\cite{MR1613330}. They do not seek for an exact asymptotic limit, focusing of the order of the fluctuations. Their methods rely on classical real analysis.

The mathematical literature is more focused on exact methods, that rely on generating functions.
The first paper in this spirit is Erd\H{o}s et al~\cite{MR869777}, which solved the case $a_n=0$ with the help of renewal equations. Tauberian theorems lead to simpler proofs of their result, see e.g.  Choimet and Queffelec~\cite{MR3445361}. Recent results by Drmota and Szpankowski~\cite{MR3078703}) also rely on Tauberian theorems and some other tools in complex analysis. They request some assumptions of monotonicity.

If one wants to cover the case of a random $(a_n)$, the sequence $(a_n)$ obviously can not be supposed to be monotonic.
Quite surprisingly, we did not find in the literature any theorem of this kind, computing an exact limit without making some assumption of monotonicity.

Let us clarify the assumptions: we assume that the $b_i$'s are positive numbers with $\sum_{j=1}^m b_j>1$, that the $m_i$ are integers with $m_i\ge 2$ and such that there exists $j,\ell$ with $\frac{\log m_j}{\log m_{\ell}}\not\in\bbQ$. 
The rational case, which is not considered here, is also of great interest in Computer Science -- see e.g. Roura~\cite{MR1868714} or  Drmota and Szpankowski~\cite{MR3078703}.

It is known that the general growth of $(X_n)$ is governed by the value of the positive root $s_0$ for the equation
\[\sum_{j=1}^m b_j m_j^{-s}=1.\]

As said before, the originality of the present paper lies in the assumption on the $(a_n)$: under the
assumption that 
\[\sum_{n=1}^{+\infty}\frac{|a_n|}{n^{s_0}}<+\infty,\]
we prove that the sequence $\frac{X_n}{n^{s_0}}$ admits a limit $L$ when $n$ tends to infinity and give a fairly simple closed expression for it.

As we will see, this allow to apply our Theorem to a large class of random variables.
Then, the limit $L$ is a random variable, which appears as the sum of a random series.

If we specialize to the case where the $(a_n)$ are independent, then one can easily control the random fluctuations
of $L$.

\section{The deterministic Theorem}
\begin{theorem}
  Let $m\ge 1$, $(b_1,\dots,b_m)$ be a family of non-negative numbers and   $(m_1,\dots,m_m)$ be a family of integers with $m_i\ge 2$ and such that
  \begin{itemize}
  \item there exists $j,\ell$ with $\frac{\log m_j}{\log m_{\ell}}\not\in\bbQ$;
  \item $\sum_{j=1}^m b_j>1$.
  \end{itemize}
  We denote by $s_0$ the positive root $s_0$ for the equation
\[\sum_{j=1}^m b_j m_j^{-s}=1.\]
  Then, there exists a sequence $(\ell_j)_{j\ge 0}$ of positive numbers such that for every sequence $(a_n)_{n\ge 0}$ with
  \[\sum_{n=1}^{+\infty}\frac{|a_n|}{n^{s_0}}<+\infty,\]
  then the sequence $(X_n)_{n\ge 0}$ defined by $X_0=a_0$ and the recursion~\eqref{recurr} satisfies
  \[\lim_{n\to +\infty} \frac{X_n}{n^{s_0}}=\sum_{j=0}^{+\infty} \ell_j a_j.\]
\end{theorem}

Note that if the sequence $(a_j)_{j\ge 0}$ is non-negative and not identically zero, the limit $\sum_{j=0}^{+\infty} \ell_j a_j$ is positive, so we have found the correct speed for the growth of $(X_n)_{n\ge 0}$. 

\begin{proof}

We denote by $L_n(a)$ the value of $X_n$ corresponding to  the recursion~\eqref{recurr} for some sequence $a$.

\subsection*{The recursion equation}
Let $n_0$ be a non-negative integer and suppose first that $a_n=0$ for $n>n_0$. 

For $n>n_0$, we have $X(n)=\sum_{j=1}^m b_j X(\floor{\frac{n}{m_j}})$.

We can choose $C$ such that $|X_k|\le Ck^{s_0}$ for $0<k\le n_1=\max(n_0,m_1,\dots,m_m)$.
Then, it follows by natural induction that  $| X_k|\le C k^{s_0}$ for each $k\in\bbN^*$.
In the sequel, we put $X(t)=X(\floor{t})$ to simplify some notation.
Now define
\begin{align}
  \label{redef}
  \phi(s)&=s\int_{n_0+1}^{+\infty} \frac{X(t)}{t^{s+1}}\ dt
\end{align}

for $s\in\bbC$ with $\Reelle(s)>s_0$. The recursion Equation leads to

\begin{align*}\phi(s)&=s\int_{n_0+1}^{+\infty}  \sum_{j=1}^m b_j \frac{X(\frac{t}{m_j})}{t^{s+1}} \ dt=s\sum_{j=1}^m b_jm_j^{-s}\int_{\frac{n_0+1}{m_j}}^{+\infty} \frac{X(t)}{t^{s+1}}\ dt\\
  &=\left(\sum_{j=1}^m b_jm_j^{-s}\right)\phi(s)+s\sum_{j=1}^m b_jm_j^{-s}\int_{\frac{n_0+1}{m_j}}^{n_0+1} \frac{X(t)}{t^{s+1}}\ dt.
  \end{align*}

Since
\begin{align*}
  |\sum_{j=1}^m b_jm_j^{-s}|\le \sum_{j=1}^m |b_jm_j^{-s}|= \sum_{j=1}^m b_jm_j^{-\Reelle(s)}
  <\sum_{j=1}^m b_jm_j^{-s_0}=1,
  \end{align*}
we can write, for $\Reelle(s)>s_0$:
\begin{align}
  \label{phitrouve}
  \phi(s)&=\frac{P(s)}{1-\sum_{j=1}^m b_jm_j^{-s}},\text{ with }P(s)=s\sum_{j=1}^m b_jm_j^{-s}\int_{\frac{n_0+1}{m_j}}^{n_0+1}\frac{X(t)}{t^{s+1}}\ dt
\end{align}

\subsection*{Tauberian magic}
Now, fix a non-negative integer $n_0$ and suppose that the sequence $a=(a_n)_{n\ge 0}$ is $a=I^{n_0}$ with \[I^{n_0}_i=\1_{i\le n_0}=\begin{cases}1 &\text{ if }i\le n_0\\ 0 &\text {else}\end{cases}.\]
By natural induction, it is easy to see that $(X_n)_{n\ge 0}$ is non-decreasing.

It is also not difficult to see that
$1-\sum_{j=1}^m b_jm_j^{-s}$ does not vanish for $s\in\bbC$ with $\Reelle(s)\ge s_0$ and $s\ne s_0$.
Proceeding as in Choimet and Queffelec (see~\cite{MR3445361}, section 4), we can note that, for $\Reelle(s)=s_0$
\begin{align*}
  \Reelle\left(  \sum_{j=1}^m b_jm_j^{-s}\right)= \sum_{j=1}^m b_jm_j^{-s_0}\cos(\log m_j \Imaginaire(s))\le\sum_{j=1}^m b_jm_j^{-s_0}=1.
\end{align*}
In fact, the inequality in strict when $\Imaginaire(s)\ne 0$. Overwise, we would have
$\log m_j \Imaginaire(s)\in 2\pi\bbZ$ for each $j$, whence  $\frac{\log m_j}{\log m_k}\in\bbQ$ for each $j,k$, which has been excluded.

It follows that for \[c=\text{Res}_{s_0}\phi=\frac{P(s_0)}{\sum_{j=1}^m b_jm_j^{-s_0}\log(m_j)},\] the map $s\mapsto \phi(s)-\frac{c}{s-s_0}$ is holomorphic on $\{s\in\bbC;\Reelle(s)\ge s_0\}$.

Now note
$b(x)=\sum_{n_0<n\le x} (X_n-X_{n-1})$.
The Abel transformation gives
\[\sum_{n=n_0+1}^{+\infty}\frac{X_n-X_{n-1}}{n^s}=s\int_{n_0+1}^{+\infty} \frac{b(t)}{t^{s+1}}\ dt.\]
Since $b(t)=X(t)-X_{n_0}$, we have
\begin{align*}\sum_{n=n_0+1}^{+\infty}\frac{X_n-X_{n-1}}{n^s}&=s\int_{n_0+1}^{+\infty} \frac{X(t)}{t^{s+1}}\ dt-\frac{X_{n_0}}{(n_0+1)^s}
  =\phi(s)-\frac{X_{n_0}}{(n_0+1)^s}.
\end{align*}

Now, we will apply the Ikehara--Newman Theorem for series:

\begin{proposition} Let $(u_n)_{n\ge 1}$ be a sequence
of non-negative real numbers, and $a$, $c$ be positive real numbers. Suppose that
the Dirichlet series $\Phi(s)=\sum_{n=1}^{+\infty}u_n n^{-s}$
is defined on the open half-plane $\Reelle (s) > a$ and that, more precisely, with
$A(x) =\sum_{n\le x} u_n$ for $x\ge 0$,
the following properties are verified:
\begin{itemize}
\item  $A(x)x^{-a}$ is bounded on $\bbR^+$ ;
\item $\Phi(s)-\frac{c}{s-a}$ has a holomorphic extension G on the closed half-plane $\Reelle(s)\ge a$.
\end{itemize}
Then we have $A(x) \sim \frac{c}a  x^a$ as $x\to +\infty$.
\end{proposition}

Since $(X_n)_{n\ge 0}$ is non-decreasing, the sequence $(X_n-X_{n-1})_{n>n_0}$ is non-negative, so the Wiener-Ikehara Theorem for series applies: since
$b(t)=O(t^{s_0})$ when $t\to +\infty$, we get $b(x)\sim \frac{c}{s_0}x^{s_0}$,so

\begin{align*}
  \lim_{n\to +\infty}\frac{L_n(I^{n_0})}{n^{s_0}}=\frac{\sum_{j=1}^m b_jm_j^{-s_0}\int_{\frac{n_0+1}{m_j}}^{n_0+1}\frac{L_t(I^{n_0})}{t^{s_0+1}}\ dt}{\sum_{j=1}^m b_jm_j^{-s_0} \log(m_j)}.
\end{align*}

For $n_0=0$, we have $\displaystyle \ell_0=\lim_{n\to +\infty}\frac{L_n(\delta^{0})}{n^{s_0}}=\lim_{n\to +\infty}\frac{L_n(I^{0})}{n^{s_0}}$, so

\begin{align*}\ell_0&=\frac1{\sum_{j=1}^m b_jm_j^{-s_0}\log (m_j)}\sum_{j=1}^m b_jm_j^{-s_0}\int_{\frac1{m_j}}^{1}\frac1{t^{s_0+1}}\ dt\quad\text{or}\end{align*}
\begin{align*}
\ell_0 &=\frac1{\sum_{j=1}^m b_jm_j^{-s_0}\log (m_j)}\sum_{j=1}^m b_jm_j^{-s_0} \frac{m_j^{s_0}-1}{s_0}.
  \end{align*}

Note that this equality and the related convergence form the result by  Erd\H{o}s et al~\cite{MR869777}.

Let $n_0\ge 1$. The sequence $(\delta^{n_0}_n)_{n\ge 0}$ is defined by
\[\delta^{n_0}_n=\begin{cases}1&\text{if }n=n_0\\ 0&\text{else}\end{cases}.\]
  Since $\delta^{n_0}=I^{n_0}-I^{n_0-1}$, it follows that
  \[L_n(\delta^{n_0})n^{-s_0}= L_n(I^{n_0})n^{-s_0}- L_n(I^{n_0-1})n^{-s_0}\]  has a limit when $n$ tends to infinity. Let us denote it by $\ell_{n_0}$.

To compute it, take $a=\delta^{n_0}$ and consider again the associated $\phi$.
  From~\eqref{redef}, we get $\ell_{n_0}=\lim_{s\to s_0^+}\frac1{s_0} (s-s_0)\phi(s)$.
  On the other side, Equation~\eqref{phitrouve} is still valid, with 
  \begin{align*}P(s)&=s\sum_{j=1}^m b_jm_j^{-s}\int_{\frac{n_0+1}{m_j}}^{n_0+1}\frac{X(t)}{t^{s+1}}\ dt\\
    &=s\sum_{j=1}^m b_jm_j^{-s}\int_{\max(n_0,\frac{n_0+1}{m_j})}^{n_0+1}\frac1{t^{s+1}}\ dt,
  \end{align*}
  also \[\frac1{s_0}(s-s_0)\phi(s)=-\frac{s}{s_0}\frac{s_0-s}{1-\sum_{j=1}^m b_j m_j^{-s}}\sum_{j=1}^m b_jm_j^{-s}\int_{\max(n_0,\frac{n_0+1}{m_j})}^{n_0+1}\frac1{t^{s+1}}\ dt\]
  and, considering that $m_j\ge 2$, we get
  
  \[\ell_{n_0}=\frac1{\sum_{j=1}^m b_jm_j^{-s_0}\log (m_j)}\sum_{j=1}^m b_jm_j^{-s_0}\int_{n_0}^{n_0+1}\frac1{t^{s_0+1}}\ dt.\]

  Thanks to this expression and the previous one, it is clear that
  $\ell_j>0$ holds for each $j\ge 0$.
  \subsection*{The general case}

For $n,j\ge 0$, we note $K^j_n=L_n(\delta^j)$.
It is obvious that $K^j_n=0$ for $n<j$ and $K^j_j=1$.
It easily follows by natural induction on $n$ that $0\le K^j_n\le\frac{K^0_n}{K^0_j}$.
Now, the affine nature of the recursion gives

\begin{align*}
  X_{n}&=\sum_{j=0}^n K^j_n a_j.
\end{align*}

For each $j\ge 0$, we have $\displaystyle \lim_{n\to +\infty}\frac{K_n^j}{n^{s_0}}=\ell_j$. Also, the $K^0_j$'s are positive, with $\displaystyle \lim_{j\to +\infty}\frac{K_j^0}{j^{s_0}}=\ell_0>0$, so there exists $M$ such that $0<\frac1{K_j^0}\le \frac{M}{j^{s_0}}$ for each $j\ge 1$
Then, for each $j,n\ge 1$, we have
\begin{align*}
  |\frac{K^j_n a_j}{n^{s_0}}|&\le \frac{K_n^0}{n^{s_0}}\frac{|a_j|}{K^0_j}\le \frac{|a_j|}{K^0_j}\le M\frac{|a_j|}{j^{s_0}}
\end{align*}

  and by the Weierstrass criterion,
  \[\lim_{n\to +\infty} \frac{X_n}{n^{s_0}}=\sum_{j=0}^{+\infty} \ell_j a_j.\]
\end{proof}
  \section{Application to sequences of random variables}

  We give below some applications of Theorem~1 to sequences of random variables.

\subsection{Convergence}

\begin{theorem}Assume that the $m_i$'s, the $b_i$'s and $s_0$ fulfill the assumptions of Theorem~1 and
  $(a_n)$ is a sequence of random variables.
  Under each of the following sets of supplementary assumptions, the sequence $(X_n)_{n\ge 0}$ defined by $X_0=a_0$ and the recursion~\eqref{recurr} is such that
  $\frac{X_n}{n^{s_0}}$ almost surely converges to some random variable,
  given as the sum of the random series:
  \[L=\sum_{j=0}^{+\infty} \ell_j a_j.\]
  \begin{enumerate}
    \item[(A)] $\sum_{j=1}^m \frac{b_j}{m_j}>1$ and the $(a_n)$ are integrable random variables with \[C=\sup_{n\ge 1}\bbE|a_n|<+\infty.\]
    \item[(B)] $\sum_{j=1}^m \frac{b_j}{m_j^2}>1$ and there exists $C>0$ such that
 for each $n\ge 1$ and $t\ge 1$, we have $\bbP(|a_n|>t)\le\frac{C}t$.
  \end{enumerate}
\end{theorem}
\begin{proof}
 \begin{enumerate}
    \item[(A)] the condition $\sum_{j=1}^m \frac{b_j}{m_j}>1$ implies that $s_0>1$.\\ We have
  $\bbE(\sum_{n=1}^{+\infty} \frac{|a_n|}{n^{s_0}})\le C\zeta(s_0)<+\infty$, so
  $\sum_{n=1}^{+\infty} \frac{|a_n|}{n^{s_0}}<+\infty$ almost surely, which gives the almost sure behavior of $\frac{X_n}{n^{s_0}}$.
\item[(B)]  the condition $\sum_{j=1}^m \frac{b_j}{m_j^2}>1$ implies that $s_0>2$. We fix $\eta>1$ with $s_0-\eta>1$.
  Then $\bbP(|a_n|>n^{\eta})=O(n^{-\eta})$ and $\sum_{n=1}^{+\infty}\bbP(|a_n|>n^{\eta})<+\infty$, so by the Borel-Cantelli Lemma,
  for almost every $\omega$, there exists $n_0(\omega)$ with  $|a_n(\omega)|\le n^{\eta}$ for $n\ge n_0(\omega)$, which
  gives the convergence of $\sum_{n\ge 1}\frac{|a_n|}{n^{s_0}}$ and our Master Theorem still applies.
  \end{enumerate}
\end{proof}

\subsection{Non-vanishing limit}

We have already noticed that the limit does not vanish when the $a_j$ are non-negative.
In the case of random independent $a_n$, it is very unlikely that the limit is null, even for signed variables.

\begin{theorem}

  Assume that the $a_i$'s, $m_i$'s, the $b_i$'s and $s_0$ fulfill the assumptions of Theorem~2 and also that
  $(a_n)$ is a sequence of independent random variables, with at least one ${j_0}\ge 0$ such that $a_j$ is non-atomic.
  Then, the limit $L=\sum_{j=0}^{+\infty} \ell_j a_j$ is non-atomic, and particularly
  $\bbP(L=0)=0$.
\end{theorem}
\begin{proof}
  By independence, the characteristic function of $L$ satisfies
  \[\forall t\in\bbR \quad |\phi_{L}(t)|=\prod_{j=0}^{+\infty}|\phi_{\ell_j a_j}(t)|\le |\phi_{\ell_{j_0} a_{j_0}}(t)|.\]

  Therefore

  \[\lim_{T\to +\infty}\frac1{2T}\int_{-T}^T |\phi_L(t)|^2\ dt\le\lim_{T\to +\infty}\ \frac1{2T}\int_{-T}^T |\phi_{\ell_{0}a_{j_0}}(t)|^2\ dt=0,\]
  which implies that $L$ is non-atomic (see e.g. Durrett~\cite{MR3930614}, section 3.3).
\end{proof}  
\subsection{Exponential moments}

\begin{theorem}
  Assume that the $m_i$'s, the $b_i$'s and $s_0$ fulfill the assumptions of Theorem~1 and
  $(a_n)$ is a sequence of independent random variables.
  The sequence $(X_n)_{n\ge 0}$ is defined by $X_0=a_0$ and the recursion~\eqref{recurr}.
  \begin{itemize}
  \item If there exists a distribution $\mu$ with exponential moments such that $|a_n|$ is stochastically dominated by $\mu^{*n}$ for each $n\ge 0$, then $|X_n|$ has exponential moments for each $n$.
  \item If $s_0>1$ (or equivalently $\sum_{j=1}^m \frac{b_j}{m_j}>1$) and there exists a distribution $\mu$ with exponential moments such that $|a_n|$ is stochastically dominated by $\mu$ for each $n\ge 0$, then  $\frac{X_n}{n^{s_0}}\to L$ a.s. where $|L|$ has exponential moments.
   \item If $s_0>2$ (or equivalently $\sum_{j=1}^m \frac{b_j}{m_j^2}>1$) and there exists a distribution $\mu$ with exponential moments such that $|a_n|$ is stochastically dominated by $\mu^{*n}$ for each $n\ge 0$, then $\frac{X_n}{n^{s_0}}\to L$ a.s. where $|L|$ has exponential moments.
  \end{itemize}
\end{theorem}

\begin{proof}
 We begin with an easy lemma: 
\begin{lemma}
Let $X$ be a random variable with  $\bbE(e^{\alpha X})<+\infty$ and $Y$ a random variable
following the exponential law  $\mathcal{E}(\alpha)$ 
Then, for  $a=\frac1{\alpha}\log \bbE (e^{\alpha X_1})$, we have the stochastic domination
$X\prec Y+a$.
\end{lemma}
\begin{proof}
We just have to prove that for $t\in\bbR$, $\bbP(X\ge t)\le \bbP(Y+a\ge t)$,
or equivalently $\bbP(X\ge t)\le \bbP(Y\ge t-a)$.
For $t\le a$, we have $\bbP(X\ge t)\le 1=\bbP(Y\ge t-a)$.
For $t\ge a$, the Markov inequality gives
\[\bbP(X\ge t)\le\frac{\bbE e^{\alpha X}}{e^{\alpha t}}=\frac{e^{\alpha a}}{e^{\alpha t}}=\exp(-\alpha (t-a))=P(Y\ge t-a).\]
This completes the proof.
\end{proof}
Now, we have $a$ and $\alpha$ such that for each $n\ge 1$

\[|a_n|\prec \mu^{* n} \prec (\delta^a*\mathcal{E}(\alpha))^{*n}=\delta^{na}*\Gamma(n,\theta).\]

Let $(Z_n)_{n\ge 0}$ be a sequence of independent  variables with $Z_n\sim \Gamma(n,\theta)$, where $\Gamma(a,\gamma)$ is the Law with the density
\[x\mapsto \frac{\gamma^a}{\Gamma(a)}x^{a-1}e^{-\gamma x}\ \1_{]0,+\infty[}(x).\]

$\frac{|X_n|}{n^{s_0}}$ is stochastically dominated by

\[M\sum_{j=0}^n \frac{ja+Z_j}{(j+1)^{s_0}},\]
so for $t<1/\alpha$, we have

\begin{align*}\bbE(e^{t\frac{|X_n|}{n^{s_0}}})&\le \exp(Ma \sum_{j=1}^{n+1} j^{-s_0})\prod_{j=0}^n\bbE\exp(\frac{tZ_j}{(j+1)^{s_0}})\\
  &\le  \exp(Ma \sum_{j=1}^{n+1} j^{-s_0})\prod_{j=0}^n (1-\frac{\alpha t}{(j+1)^{s_0}})^{-j}.
\end{align*}
When $j$ is large enough, $(1-\frac{\alpha t}{(j+1)^{s_0}})^{-j}\le\exp(\frac{\alpha t}{j^{s_0-1}})$, which gives the existence of an exponential moment  for $s_0>2$.

The proof in the case $|a_n|\prec \mu$ and $s_0>1$ is similar.

\end{proof}  
 
As an example of domination by $\mu^{*n}$, we can think about the case where a recursive function called with parameter $n$
requires $n$ simulations with an acceptance-rejection method. Then, $a_n$ appears as the sum of $n$ independent variables following a geometric distribution $\mu=\mathcal{G}(p)$.
%
\def\refname{References}
\bibliographystyle{plain}
\bibliography{divide}

\begin{thebibliography}{1}

\bibitem{MR1613330}
Mohamad Akra and Louay Bazzi.
\newblock On the solution of linear recurrence equations.
\newblock {\em Comput. Optim. Appl.}, 10(2):195--210, 1998.

\bibitem{MR3445361}
D.~Choimet and H.~Queff\'{e}lec.
\newblock {\em Twelve landmarks of twentieth-century analysis}.
\newblock Cambridge University Press, New York, 2015.
\newblock Illustrated by Micha\"{e}l Monerau, Translated from the 2009 French
  original by Dani\`ele Gibbons and Greg Gibbons, With a foreword by Gilles
  Godefroy.

\bibitem{cormen}
Thomas~H. Cormen, Charles~E. Leiserson, Ronald~L. Rivest, and Stein.
\newblock {\em Introduction to Algorithms}.
\newblock PHI Learning, 2010.

\bibitem{MR3078703}
Michael Drmota and Wojciech Szpankowski.
\newblock A master theorem for discrete divide and conquer recurrences.
\newblock {\em J. ACM}, 60(3):Art. 16, 49, 2013.

\bibitem{MR3930614}
Rick Durrett.
\newblock {\em Probability---theory and examples}, volume~49 of {\em Cambridge
  Series in Statistical and Probabilistic Mathematics}.
\newblock Cambridge University Press, Cambridge, 2019.
\newblock Fifth edition of [ MR1068527].

\bibitem{MR869777}
P.~Erd\H{o}s, A.~Hildebrand, A.~Odlyzko, P.~Pudaite, and B.~Reznick.
\newblock The asymptotic behavior of a family of sequences.
\newblock {\em Pacific J. Math.}, 126(2):227--241, 1987.

\bibitem{goodrich}
Michael~T. Goodrich and Roberto Tamassia.
\newblock {\em Algorithm Design: Foundations, Analysis, and Internet Examples}.
\newblock Wiley, 2002.

\bibitem{MR1868714}
Salvador Roura.
\newblock Improved master theorems for divide-and-conquer recurrences.
\newblock {\em J. ACM}, 48(2):170--205, 2001.

\end{thebibliography}
\end{document}